\documentclass[a4paper,12pt,reqno]{amsart}
\usepackage[a4paper,margin=1.9cm,top=2.5cm,bottom=2.5cm,centering,vcentering]{geometry}
\usepackage{amsmath,amssymb,graphicx,amsfonts,tikz}
\usetikzlibrary{arrows,matrix}
\usepackage{amsfonts,hyperref} 
\usepackage{amsthm}
\newtheorem{theorem}{Theorem}[section]

\usepackage[utf8]{inputenc}
\usepackage[english]{babel}
\DeclareMathOperator{\ped}{ped}
\DeclareMathOperator{\podd}{pod}
\DeclareMathOperator{\de}{DE1}
\DeclareMathOperator{\dee}{DE2}
\DeclareMathOperator{\deee}{DE3}
\DeclareMathOperator{\ddo}{DO1}
\DeclareMathOperator{\doo}{DO2}
\DeclareMathOperator{\dooo}{DO3}
\usepackage{ulem}

\title[Combinatorial Proofs of Some Results of Andrews and El Bachraoui]{Combinatorial Proofs of Some Results of Andrews and El Bachraoui}

       \author[P. J. Mahanta]{Pankaj Jyoti Mahanta}
    \address[P. J. Mahanta]{Gonit Sora, Dhalpur, Assam 784165, India}
    \email{pjm2099@gmail.com}
    
\author[M. P. Saikia]{Manjil P. Saikia}
\address[M. P. Saikia]{Mathematical and Physical Sciences division, School of Arts and Sciences, Ahmedabad University, Ahmedabad 380009, Gujarat, India}
\email{manjil@saikia.in}

\date{\today}


\begin{document}

\begin{abstract}
Recently, Andrews and El Bachraoui (2024) proved three very interesting $q$-series identities, from which three simple looking identities involving certain restricted partitions into distinct even parts and $4$-regular partitions follow. In this short note, we give combinatorial proofs of these identities. We also prove the counterpart identities for the restricted partitions into distinct odd parts.
\end{abstract}

\maketitle
\section{Introduction}

A partition of $n$ is a non-increasing sequence $\lambda=(\lambda_1, \lambda_2, \ldots, \lambda_k)$ such that $\lambda_1\geq \lambda_2\geq \cdots \geq \lambda_k$ and $\sum\limits_{i=1}^k\lambda_i = n$. For instance the $7$ partitions of $5$ are
\[
5, 4+1, 3+2, 3+1+1, 2+2+1, 2+1+1+1, 1+1+1+1+1.
\]
Over the last several decades, various mathematicians have studied many different aspects of partitions as well as several restricted classes of partitions. For a general survey of the theory of partitions, we refer the reader to the books of Andrews \cite{gea1} and Johnson \cite{Johnson}.

In this paper we are interested in the following restricted sets of partitions
\begin{itemize}
    \item $P_{\ped}(n)=$ the set of partitions of $n$ with distinct even parts,
    \item $P_{\ped{>1}}(n)=$ the subset of partitions in $P_{\ped}(n)$ such that parts are strictly larger than $1$,
    \item $D1(n)=$ the set of partitions of $n$ in which no even part is repeated and the largest part is odd,
    \item $D2(n)=$ the set of partitions of $n$ in which no even is repeated, and the largest part is odd and it appears at least twice, and
    \item $D3(n)=$ the set of partitions of $n$ in which no even part is repeated, and the largest part is odd and appears exactly once.
\end{itemize}
We denote the cardinalities of the above sets by $\ped(n)$, $\ped_{>1}(n)$, $\de(n)$, $\dee(n)$ and $\deee(n)$ respectively.

Recently, Andrews and El Bachraoui \cite{AndrewsElBachraoui} proved the following very simple looking but intriguing formulas, using $q$-series techniques.

\begin{theorem}[Corollary 1, \cite{AndrewsElBachraoui}]\label{cor1}
    For $n>0$, we have
    \[
    \de(n)+\de(n-1)=\ped(n).
    \]
\end{theorem}

\begin{theorem}[Corollary 2, \cite{AndrewsElBachraoui}]\label{cor2}
    For all $n>0$, we have
    \[
    \dee(n)+\dee(n-3)=\ped_{>1}(n).
    \]
\end{theorem}

\begin{theorem}[Corollary 3, \cite{AndrewsElBachraoui}]\label{cor3}
    For all $n>0$, we have
    \[
    \deee(n+2)+\deee(n-1)=\ped(n).
    \]
\end{theorem}

\noindent At the end of their paper, Andrews and El Bachraoui \cite{AndrewsElBachraoui} asked for combinatorial proofs of the above results. We give such proofs in this paper, in Section \ref{sec:two}. We note that, Andrews and El Bachraoui \cite{AndrewsElBachraoui} gave the right hand sides of the above results in terms of $4$-regular partitions, which are equinumerous with partitions into even parts distinct. We also note that Chen and Zou \cite{ChenZou} also proved these identities combinatorially, but their proofs are different from our proofs.

We can also ask for the corresponding results if we look at restricted partitions where all odd parts are distinct. To that end, we define the following sets:
\begin{itemize}
    \item $P_{\podd}(n)=$ set of all partitions of $n$ with distinct odd parts,
    \item $P_{\podd>2}(n)=$ subset of $P_{\podd}(n)$ where all parts are greater than $2$,
    \item $O1(n)=$ set of partitions of $n$ in which no odd part is repeated and the largest part is even,
    \item $O2(n)=$ subset of $O1(n)$ where the largest part appears at least twice, and
    \item $O3(n)=$ subset of $O1(n)$ where the largest part appears exactly once.
\end{itemize}
We denote the cardinalities of the above sets by $\podd(n)$, $\podd_{>2}(n)$, $\ddo(n)$, $\doo(n)$ and $\dooo(n)$ respectively. Similar to the results of Andrews and El Bachraoui \cite{AndrewsElBachraoui}, we have the following results.

\begin{theorem}\label{pod-cor1}
    For all $n>1$, we have
    \[
    \ddo(n)+\ddo(n-1)=\podd(n).
    \]
\end{theorem}

\begin{theorem}\label{pod-cor2}
    For all $n>4$, we have
    \[
    \doo(n)+\doo(n-3)=\podd_{>2}(n).
    \]
\end{theorem}

\begin{theorem}\label{pod-cor3}
    For all $n>2$, we have
    \[
    \dooo(n+2)+\dooo(n-1)=\podd(n).
    \]
\end{theorem}

\noindent The proofs of Theorems \ref{pod-cor1} and \ref{pod-cor3} are similar to the proofs of Theorems \ref{cor1} and \ref{cor3} respectively, so we omit them here. We prove Theorem \ref{pod-cor2} in Section \ref{sec:three}. It is natural to also look for $q$-series proofs of the above results, we leave that as an open problem.

\section{Combinatorial Proofs of Theorems \ref{cor1}, \ref{cor2} and \ref{cor3}}\label{sec:two}

\begin{proof}[Proof of Theorem \ref{cor1}]
    We start with a $\lambda \in D1(n-1)$ and add $1$ to the largest part (if there are repeated largest parts, then we add $1$ to only one of those parts), resulting in the partition $\mu$, which is a partition of $n$ in which no even part is repeated and the largest part is even. Clearly, the largest part of $\mu$ is even and appears only once, thus $\mu$ is a partition of $n$ with no even parts being repeated. So, $\de(n-1)$ counts the number of partitions in $P_{\ped}(n)$ where the largest part is even, and $\de(n)$ counts the number of partitions in $P_{\ped}(n)$ where the largest part is odd. Thus,
    \[
    |D1(n)\cup D1(n-1)|=|P_{\ped}(n)|,
    \]
    and this completes the proof.
\end{proof}

\begin{proof}[Proof of Theorem \ref{cor2}]
   We start by noticing that the partitions in $D2(n)$ are of the following form: $\lambda = (2j+1, 2j+1, h, \ldots)$ for some $h,j$ with $h\leq 2j+1$. Adding $2$ and $1$ to the first two parts of a partition $\lambda \in D2(n-3)$ results in a partition of the form $(2j^\prime+1, 2j^\prime, h^\prime, \ldots)$ with $h^\prime \leq 2j^\prime -1$. Both are clearly partitions in $P_{\ped}(n)$. The other type of partitions in $P_{\ped}(n)$ are of the following types:
   \begin{itemize}
       \item $(2\ell, h, \ldots)$, with $h\leq 2\ell-1$, and
       \item $(2\ell+1, h, \ldots)$, with $h\leq 2\ell-1$.
   \end{itemize}

   Let us define the following sets of partitions
   \begin{itemize}
       \item $A=\{(2\ell+1, 2\ell+1, h, \ldots)\in D2(n): h\leq 2\ell+1~\text{and, at least one $1$ is a part}\}$,
         \item $B=\{(2\ell+1, 2\ell, h, \ldots)\in P_{\ped}(n): h\leq 2\ell-1~\text{and, at least one $1$ is a part}\}$,
      \item $C=\{(2\ell, h, \ldots)\in P_{\ped}(n): h\leq 2\ell-1~\text{and, no $1$ is a part}\}$, and
      \item $D=\{(2\ell+1, h, \ldots)\in P_{\ped}(n): h\leq 2\ell-1~\text{and, no $1$ is a part}\}$.
   \end{itemize}

We construct a map from $C$ to $A$ in the following way:
\begin{equation}
    (2\ell, h, \ldots)\rightarrow \begin{cases}
        (h,h,\ldots, \underbrace{1, 1, \ldots, 1}_{2\ell-h}), &\text{if $h$ is odd},\\
        (h-1, h-1, \ldots, \underbrace{1,1, \ldots, 1}_{2\ell-h+2}), & \text{if $h$ is even and $h\neq 2$}.
    \end{cases}
\end{equation}
If $h$ is odd, then $2\ell-h$ is odd and $\geq 1$, while if $h$ is even then $2\ell-h+2$ is even and $\geq 4$. Thus, partitions in $A$ with exactly two $1$'s are not yet mapped, as is the partition $(1,1,\ldots, 1)\in A$ not mapped. On the other hand, if $n$ is even then $(n), (n-2, 2)\in C$ and are not mapped yet. We define
\begin{itemize}
    \item $A^\prime=\{\lambda\in A: \text{the part $1$ appears exactly twice}\}\cup \{(1, 1, \ldots, 1)\}$, and
    \item $C^\prime=\{(n), (n-2,2)\}$, with $n$ even.
\end{itemize}

Like before, we construct a map from $D$ to $B$. We assume $2\ell+1-h\neq 2$.
\begin{equation}
    (2\ell+1, h, \ldots)\rightarrow \begin{cases}
        (h+2, h+1, \ldots, \underbrace{1,1,\ldots, 1}_{2\ell-h-2}), & \text{if $h$ is odd},\\
         (h+1, h, \ldots, \underbrace{1,1,\ldots, 1}_{2\ell-h}), & \text{if $h$ is even}.
    \end{cases}
\end{equation}
If $n$ is even then $(n-2,2)\notin D$, so the partition of the form $(3,2,1,\ldots, 1)\in B$ is not mapped in this case. We define
\begin{itemize}
    \item $B^\prime=\{(3,2,1,\ldots,1)\}$, with $n$ even, and
    \item $D^\prime=\begin{cases}
        \{\lambda\in D: 2\ell+1-h=2\}\cup \{(n)\}, \quad \text{if $n$ is odd},\\
        \{\lambda\in D:2\ell+1-h=2\}, \quad \text{if $n$ is even}.
    \end{cases}$
\end{itemize}
We want to show $|C^\prime\cup D^\prime|=|A^\prime\cup B^\prime|$, which would immediately imply $|C\cup D|=|A\cup B|$. This is not difficult to show, we define the following maps:

\textbf{Case I ($n$ is odd)} The relevant map is
\[
(n)\rightarrow (\underbrace{1,1,\ldots, 1}_{n})
\]
and
\[
(2\ell+1, 2\ell-1, \ldots)\rightarrow (2\ell-1, 2\ell-1, \ldots, k, 1,1), \quad \text{with $k\geq 2$}.
\]

\textbf{Case II ($n$ is even)} The relevant map is
\[
(n)\rightarrow (\underbrace{1,1,\ldots, 1}_{n}),
\]
\[
(2\ell+1, 2\ell-1, \ldots)\rightarrow (2\ell-1, 2\ell-1, \ldots, k, 1,1), \quad \text{with $k\geq 2$},
\]
and
\[
(n-2,2)\rightarrow (3,2,\underbrace{1,1,\ldots, 1}_{n-5}).
\]
This proves the result.
\end{proof}

\begin{proof}[Proof of Theorem \ref{cor3}]
    We start with a $\lambda\in D3(n-1)$ and add $1$ to the largest part, resulting in the partition $\mu$. Clearly $\mu\in P_{\ped}(n)$ and is of the form $(2j, k, \ldots)$, for some $k, j$ with $k\leq 2j-2$. 
    
    On the other hand, the partitions in $D3(n+2)$ are of the following three types
    \begin{itemize}
        \item[(a)] $(2j+1, 2j, h, \ldots)$ for some $j, h$ with $h\leq 2j-1$.
        \item[(b)] $(2j+1, 2j-1, h, \ldots)$ for some $j,h$ with $h\leq 2j-1$.
         \item[(c)] $(2j+1, h, \ldots)$ for some $j,h$ with $h<2j-1$.
    \end{itemize}
    Now, if we subtract $2$ from the partitions above then the resulting partition is in $P_{\ped}(n)$. Clearly, all partitions in $P_{\ped}(n)$ are accounted for from the above discussion. This completes the proof.
\end{proof}

\section{Proof of Theorem \ref{pod-cor2}}\label{sec:three}

The partitions in $O2(n)$ are of the form $(2\ell, 2\ell, h, \ldots)$ with $h\leq 2\ell$. We add $2$ and $1$ respectively to the two largest parts of the partitions in $O2(n-3)$. The resulting partitions are of the form $(2\ell, 2\ell-1, h, \ldots)$ with $h\leq 2\ell-2$. This is a partition in $P_{\podd}(n)$. The remaining partitions in $P_{\podd}(n)$ are of the form $(2\ell+1, h, \ldots$) with $h\leq 2\ell$ or $(2\ell, h, \ldots)$ with $h\leq 2\ell-2$.

We now define the following sets
\begin{itemize}
    \item $A=\{(2\ell, 2\ell, h, \ldots)\in O2(n): h\leq 2\ell~\text{and $2$ or $1$ appears as a part at least once}\}$,
    \item $B=\{(2\ell, 2\ell-1, h, \ldots)\in P_{\podd}(n): h\leq 2\ell-2~\text{and $2$ or $1$ appears as a part at least once}\}$,
    \item $C=\{(2\ell+1, h, \ldots)\in P_{\podd}(n): h\leq 2\ell~\text{and each part is at least 3}\}$,
    \item $D=\{(2\ell,  h, \ldots)\in P_{\podd}(n): h\leq 2\ell-2~\text{and each part is at least 3}\}$.
\end{itemize}
We construct maps like in the proof of Theorem \ref{cor2}. We just list the relavant maps here without going the details.

Let $n$ be even, then
\[
(2\ell+1, h, \ldots)\rightarrow \begin{cases}
    (h,h,\ldots, \underbrace{2, 2, \ldots, 2}_{2\ell-h}, 1),&\text{if $h$ is even},\\
    (h-1,h-1,\ldots, \underbrace{2,2,\ldots, 2}_{2\ell+3-h}),&\text{if $h$ is odd and $h\neq 3$},\\
    (4,3,2,2,\ldots,2,1),&\text{if $h=3$},
\end{cases}
\]
\noindent and
\[
(2\ell, h, \ldots)\rightarrow \begin{cases}
(2,2,\ldots,2),&\text{if $h=0$},\\
    (h+2,h+1,\ldots, \underbrace{2, 2, \ldots, 2}_{2\ell-h-4}, 1),&\text{if $h$ is even and $2\ell-h\neq 2$},\\
    (h+1,h,\ldots, \underbrace{2,2,\ldots, 2}_{2\ell-h-1}),&\text{if $h$ is odd},\\
    (h,h,\ldots, 2),&\text{if $h$ is even and $2\ell-h=2$}.
\end{cases}
\]

Let $n$ be odd, then
\[
(2\ell+1,h, \ldots)\rightarrow \begin{cases}
    (2,2,\ldots, 2, 1),&\text{if $h=0$},\\
    (h,h,\ldots, \underbrace{2,2,\ldots,2}_{2\ell-h},1),&\text{if $h$ is even},\\
    (h-1,h-1,\ldots, \underbrace{2, 2, \ldots, 2}_{2\ell-h+3}),&\text{if $h$ is odd},
\end{cases}
\]
\noindent and
\[
(2\ell, h, \ldots)\rightarrow \begin{cases}
    (h+2,h+1, \ldots, \underbrace{2,2,\ldots,2}_{2\ell-h-4},1),&\text{if $h$ is even and $2\ell-h\neq 2$},\\
    (h+1,h,\ldots, \underbrace{2,2,\ldots,2}_{2\ell-h-1}),&\text{if $h$ is odd},\\
    (h,h,\ldots, 2),&\text{if $h$ is even and $2\ell-h=2$}.
\end{cases}
\]

\end{document}